\documentclass[11pt,a4paper]{article}

\usepackage[margin=2.5cm]{geometry}
\usepackage{amsmath,amssymb,amsthm}
\usepackage{aliascnt}
\usepackage{mathrsfs}
\usepackage{enumitem}
\usepackage{booktabs}
\usepackage[hidelinks]{hyperref}
\usepackage{cleveref}
\newtheorem{theorem}{Theorem}[section]
\newaliascnt{lemma}{theorem}
\newtheorem{lemma}[lemma]{Lemma}
\aliascntresetthe{lemma}
\newaliascnt{proposition}{theorem}
\newtheorem{proposition}[proposition]{Proposition}
\aliascntresetthe{proposition}
\newaliascnt{corollary}{theorem}
\newtheorem{corollary}[corollary]{Corollary}
\aliascntresetthe{corollary}
\theoremstyle{definition}
\newaliascnt{definition}{theorem}

\aliascntresetthe{definition}
\theoremstyle{remark}
\newaliascnt{remark}{theorem}
\newtheorem{remark}[remark]{Remark}
\aliascntresetthe{remark}

\crefname{theorem}{Theorem}{Theorems}
\Crefname{theorem}{Theorem}{Theorems}
\crefname{lemma}{Lemma}{Lemmas}
\Crefname{lemma}{Lemma}{Lemmas}
\crefname{proposition}{Proposition}{Propositions}
\Crefname{proposition}{Proposition}{Propositions}
\crefname{corollary}{Corollary}{Corollaries}
\Crefname{corollary}{Corollary}{Corollaries}
\crefname{definition}{Definition}{Definitions}
\Crefname{definition}{Definition}{Definitions}
\crefname{remark}{Remark}{Remarks}
\Crefname{remark}{Remark}{Remarks}
\crefname{equation}{equation}{equations}
\Crefname{equation}{Equation}{Equations}

\newcommand{\NN}{\mathbb{N}}
\newcommand{\EE}{\mathbb{E}}
\newcommand{\ind}{\mathbf{1}}

\title{Universal Fourier-inertia bounds for prescribed even distances}
\author{
Xiaochen Zhao\thanks{School of Mathematical Sciences, Capital Normal University, Beijing, China. Email: 2250501013@cnu.edu.cn.}
\and
Gennian Ge\thanks{School of Mathematical Sciences, Capital Normal University, Beijing, China. Email: gnge@zju.edu.cn. Gennian Ge is supported by the National Key Research and Development Program of China under Grant 2025YFC3409900, the National Natural Science Foundation of China under Grant 12231014, and Beijing Scholars Program.}
}
\date{}

\begin{document}
\maketitle

\begin{abstract}
The study of set families with restricted Hamming distances is  a classical topic of extremal combinatorics and coding theory. Let \(H=\{A\subseteq[n]: |A|\text{ is even}\}\) be the even subcube. Let \(\ell_1,\dots,\ell_t\) be distinct positive integers and set
\(\mathcal L=\{2\ell_1,\dots,2\ell_t\}\). We prove that, for all sufficiently large \(n\), every family \(\mathcal F\subseteq H\) satisfying
\(
|A\triangle B|\in\mathcal L
\)
for all \(A\ne B\in\mathcal F\)
has
\[
|\mathcal F|\le \sum_{i=0}^{t}\binom{n-1}{i}.
\]
This is best possible as a universal bound, with equality attained at the distance set \(\mathcal L=\{2,4,\dots,2t\}\).

Our proof uses a Fourier-inertia argument based on a universal low/high boundary-layer sign pattern for the Fourier transform of the distance-polynomial kernel on the even subcube: the prescribed distances enter only through lower-order Fourier terms, while the leading boundary-layer signs depend solely on \(t\). This even-subcube result immediately yields an odd-subcube analogue and, through parity reductions, provides bounds for arbitrary distance sets. In particular, this approach recovers the classical interval bounds of Kleitman and the corresponding interval bounds of Huang--Klurman--Pohoata, while offering a direct spectral proof of the maximality of \(\{2,4,\dots,2t\}\) among all fixed \(t\)-distance sets.
\end{abstract}

\section{Introduction}
The extremal theory of restricted distances in the hypercube has developed through a variety of algebraic, spectral, and intersection-theoretic methods. Kleitman's classical isodiametric theorem~\cite{Kleitman} determines the largest family of subsets of \([n]\) with bounded diameter. Huang--Klurman--Pohoata~\cite{HKP} gave an elegant algebraic proof of Kleitman's theorem using pseudo-adjacency matrices and Cvetkovi\'c's inertia bound, and extended this method to consecutive distance intervals. Dong--Gao--Liu--Ouyang--Zhou~\cite{DGLOZ} developed a restricted-intersection approach to arithmetic-progression distance sets and binary \(t\)-distance problems. In general, the binary \(t\)-distance problem asks for the maximum size of a set family whose pairwise distances are restricted to a prescribed set of \(t\) positive integers, and an important problem is to determine which \(t\)-element distance sets achieve this maximum.

A natural setting for exploring these restricted distance problems is the even subcube, which captures the parity-preserving properties of distance-constrained families. Indeed, the even subcube is a subgroup under symmetric difference, and hence the distance between any two of its elements is even. To set up the model, we identify subsets of \([n]=\{1,\dots,n\}\) with their indicator vectors in the Boolean cube, and write
\[
H=\{A\subseteq[n]: |A|\text{ is even}\}
\]
for the even subcube. Let \(\ell_1,\dots,\ell_t\) be distinct positive integers and set
\(\mathcal L=\{2\ell_1,\dots,2\ell_t\}\).
Let \(\Gamma=\Gamma(n,\mathcal L)\) be the graph on the vertex set \(H\) in which two distinct vertices \(A,B\in H\) are adjacent if and only if
\(
|A\triangle B|\notin\mathcal L.
\)
Thus, independent sets in \(\Gamma\) correspond precisely to families \(\mathcal F\subseteq H\) satisfying
\(
|A\triangle B|\in\mathcal L
\)
for all \(A\ne B\in\mathcal F\).

Our main result establishes a universal upper bound that holds uniformly over all such distance sets.

\begin{theorem}[Universal upper bound for fixed even distance sets]\label{thm:main}
Let \(\ell_1,\dots,\ell_t\) be distinct positive integers and set
\(\mathcal L=\{2\ell_1,\dots,2\ell_t\}\).
Then there exists \(n_0=n_0(\mathcal L)\) such that, for every \(n\ge n_0\), every family \(\mathcal F\subseteq H\) satisfying
\(
|A\triangle B|\in\mathcal L
\)
for all \(A\ne B\in\mathcal F\)
has
\[
|\mathcal F|\le \sum_{i=0}^{t}\binom{n-1}{i}.
\]
Moreover, for \(\mathcal L=\{2,4,\dots,2t\}\), equality is attained.
\end{theorem}

Our proof relies on a compact Fourier-inertia argument, the analytic core of which is the discovery of a universal low/high boundary-layer sign pattern for the distance-polynomial kernel on the even subcube. We identify \(H\) with the multiplicative group \(G=\{-1,1\}^{n-1}\). If we define the linear sum \(S=\sum_{i=1}^{n-1}z_i\) and the full product \(P=\prod_{i=1}^{n-1}z_i\), the corresponding even subset has Hamming distance
\(
d(\boldsymbol z)=(n-S-P)/2
\)
from the empty set. Consequently, the distance-polynomial kernel
\[
f_t(\boldsymbol z)=\prod_{j=1}^{t}\bigl(d(\boldsymbol z)-2\ell_j\bigr)
\]
can be compactly expressed in the form \(2^{-t}\prod_{j=1}^{t}\bigl(c_j-(S+P)\bigr)\), where \(c_j=n-4\ell_j\). On the Fourier characters \(\chi_T(\boldsymbol z)=\prod_{i\in T}z_i\) of \(G\), the terms \(S\) and \(P\) act as structured shift operators: multiplication by \(S\) moves frequencies between adjacent layers, while multiplication by \(P\) maps \(\chi_T\) directly to its complementary layer \(\chi_{T\triangle[n-1]}\). 
By a short induction, this algebraic mechanism reveals a stable low/high
boundary-layer profile: the distance parameters affect only lower-order
Fourier terms, while the leading signs are universal. Paired with an
appropriate diagonal correction, Cvetkovi\'c's inertia bound translates this
spectral localization directly into the extremal estimate.

This Fourier analysis can be viewed in a broader context. From the distance-regular graph viewpoint, the even subcube represents the vertex set of the halved cube, and its even Hamming-distance relations naturally form the halved-cube association scheme; see, for example,~\cite{BCN}. While our problem fits into the classical Delsarte association-scheme tradition for codes and designs~\cite{Delsarte, Delsarte-Goethals-Seidel}, the argument presented here does not rely on standard absolute-bound or rank-counting arguments. Instead, its distinctive feature is purely spectral, driving the optimal bound via the analytic behavior of the boundary-layer Fourier sign pattern.

Beyond establishing \Cref{thm:main}, this Fourier-inertia argument yields several immediate combinatorial consequences:
\begin{itemize}
    \item \textbf{Parity Extensions beyond the Even Subcube:} A parity-reversing involution \(A\mapsto A\triangle\{n\}\) yields an immediate odd-subcube analogue, while parity splitting extends the upper bounds to arbitrary (not necessarily even) distance sets.
    \item \textbf{Recovery of Classical Theorems:} In the interval case, this mechanism recovers Kleitman's classical isodiametric theorem with the sharp threshold discussed in Section~\ref{sec:thresholds} and the interval-distance bounds of Huang--Klurman--Pohoata.
    \item \textbf{Global Maximality of \(\{2,4,\dots,2t\}\):} For a fixed \(t\)-element distance set \(D\), let \(A(n,D)\) denote the maximum size of a family \(\mathcal{F} \subseteq 2^{[n]}\) whose pairwise distances lie in \(D\). Our argument provides an alternative spectral proof that the distance set \(D_0=\{2,4,\dots,2t\}\) globally maximizes \(A(n,D)\) among all fixed \(t\)-distance sets for sufficiently large \(n\).  This consequence stands complementary to the structure-sensitive results of Dong--Gao--Liu--Ouyang--Zhou~\cite{DGLOZ}; while their work addresses uniqueness and gap phenomena, our approach establishes the maximality statement via a different, analytic route.
\end{itemize}

The paper is organized as follows. Section~\ref{sec:setup} sets up the Fourier framework on the even subcube and realizes the pseudo-adjacency matrix as a convolution operator. Section~\ref{sec:fourier} proves the boundary-layer localization theorem by induction. Section~\ref{sec:inertia} applies the inertia bound and derives the extremal consequences. Section~\ref{sec:thresholds} concludes with a short discussion on effective thresholds.

\section{The even subcube and the Fourier setup}\label{sec:setup}

\paragraph{The group model.}
We identify the even subcube \(H\), viewed as an abelian group under the symmetric difference \(\triangle\), with the multiplicative group 
\[
G=\{-1,1\}^{n-1}.
\]
This identification is formalized as follows. For any vector \(\boldsymbol z=(z_1,\dots,z_{n-1})\in G\), we define its negative support by
\[
R(\boldsymbol z)=\{i\in[n-1]: z_i=-1\}.
\]
Then, \(\boldsymbol z\) corresponds uniquely to the even subset \(A(\boldsymbol z)\in H\) via the parity-matching rule:
\[
A(\boldsymbol z)=
\begin{cases}
R(\boldsymbol z), & |R(\boldsymbol z)|\text{ is even},\\
R(\boldsymbol z)\cup\{n\}, & |R(\boldsymbol z)|\text{ is odd}.
\end{cases}
\]
A direct verification shows that \(A(\boldsymbol x\boldsymbol y)=A(\boldsymbol x)\triangle A(\boldsymbol y)\) for all \(\boldsymbol x,\boldsymbol y\in G\), confirming that this mapping constitutes a group isomorphism \(G\cong H\). Throughout the paper, we shall freely interchange between these two representations.

For \(\boldsymbol z\in G\), we introduce the coordinate sum \(S\) and the full product \(P\):
\[
S=S(\boldsymbol z)=\sum_{i=1}^{n-1}z_i,
\qquad
P=P(\boldsymbol z)=\prod_{i=1}^{n-1}z_i.
\]
Let \(d(\boldsymbol z)=|A(\boldsymbol z)|\) denote the Hamming distance from the even set \(A(\boldsymbol z)\) to the identity element \(\varnothing\in H\). This Hamming distance can be algebraically evaluated as
\begin{equation}\label{eq:dist}
d(\boldsymbol z)=\frac{n-S-P}{2}.
\end{equation}
Indeed, setting \(k=|R(\boldsymbol z)|=\frac{n-1-S}{2}\), it follows that \(P=(-1)^k\). Since \(A(\boldsymbol z)=R(\boldsymbol z)\) when \(k\) is even and \(A(\boldsymbol z)=R(\boldsymbol z)\cup\{n\}\) when \(k\) is odd, we naturally obtain
\[
d(\boldsymbol z)=|A(\boldsymbol z)|
=
k+\frac{1-P}{2}
=
\frac{n-1-S}{2}+\frac{1-P}{2}
=
\frac{n-S-P}{2}.
\]

\paragraph{Fourier notation.}
For each subset \(T\subseteq[n-1]\), let \(\chi_T(\boldsymbol z)=\prod_{i\in T}z_i\) denote the standard Walsh--Hadamard character of \(G\). We equip the space of real-valued functions on \(G\) with the normalized inner product
\[
\langle g,h\rangle=\EE_{\boldsymbol z\in G}[g(\boldsymbol z)h(\boldsymbol z)],
\]
and define the Fourier coefficient of a function \(h:G\to\mathbb R\) at \(T\) by \(\widehat h(T)=\EE[h\chi_T]\). The collection of characters \(\{\chi_T\}_{T\subseteq [n-1]}\) forms an orthonormal basis for this space~\cite{ODonnell}.

\paragraph{The pseudo-adjacency matrix.}
For a graph \(\Gamma=(V,E)\), a real symmetric matrix \(M\) indexed by \(V\) is a \emph{pseudo-adjacency matrix} of \(\Gamma\) if its diagonal entries vanish and \(M_{uv}=0\) whenever \(u\ne v\) and \(uv\notin E\). We now construct such a matrix for the distance graph \(\Gamma(n,\mathcal L)\). Let
\[
f(x)=\prod_{j=1}^{t}(x-2\ell_j).
\]
For any \(A,B\in H\), we define the kernel \(F(A,B)=f(|A\triangle B|)\), and introduce the matrix \(M\) indexed by \(H\) via
\begin{equation}\label{eq:pseudo-M}
M_{A,B}=F(A,B)-(-2)^t\left(\prod_{j=1}^{t}\ell_j\right)\ind_{A=B}.
\end{equation}
Since \(f(0)=(-2)^t\prod_j\ell_j\), the diagonal entries of \(M\) are identically zero. Furthermore, if \(A\ne B\) and \(|A\triangle B|\in\mathcal L\), the polynomial factor vanishes, yielding \(F(A,B)=0\). Consequently, \(M\) serves as a valid pseudo-adjacency matrix for \(\Gamma(n,\mathcal L)\).

\paragraph{The convolution operator and eigenvalues.}
The structural symmetry of \(M\) allows for a seamless spectral decomposition via Fourier characters. Transferring the action of \(M\) from \(H\) to the group \(G\), we observe that the matrix entry \(M_{x,y}\) depends solely on the group element \(xy\). Thus, there exists a function \(h_g:G\to\mathbb R\) such that
\[
M_{x,y}=h_g(xy), \quad \text{for all } x,y \in G.
\]
Under this setup, the action of \(M\) on any function \(\varphi:G\to\mathbb R\) is given by
\[
(M\varphi)(x)=\sum_{y\in G}h_g(xy)\varphi(y).
\]
Consequently, as an operator on the function space \(\mathbb{R}^{G}\), the pseudo-adjacency matrix \(M\) acts exactly as the group convolution operator with kernel \(h_g\). 

As \(G\) is a finite abelian group, the translation-invariant operator \(M\) is diagonalized by the characters \(\chi_T\). Specifically,
\[
(M\chi_T)(x)
=
\sum_{y\in G}h_g(xy)\chi_T(y)
=
\chi_T(x)\sum_{z\in G}h_g(z)\chi_T(z),
\]
where we substituted \(z=xy\) and utilized the self-inverse property \(x^{-1}=x\) in \(G\) alongside the multiplicative property \(\chi_T(xz)=\chi_T(x)\chi_T(z)\). Under our normalized Fourier transform, the eigenvalue associated with the character \(\chi_T\) is given by
\begin{equation}\label{eq:eigenvalues}
\lambda_T=2^{n-1}\widehat h_g(T).
\end{equation}
Crucially, \eqref{eq:eigenvalues} implies that the signs of the eigenvalues of \(M\) coincide exactly with the signs of the Fourier coefficients \(\widehat h_g(T)\).

To complete the setup, let \(u(\boldsymbol z)=S(\boldsymbol z)+P(\boldsymbol z)\) and \(c_j=n-4\ell_j\). By relation \eqref{eq:dist}, the distance-polynomial kernel evaluates to
\begin{equation}\label{eq:ft}
f_t(\boldsymbol z):=f(d(\boldsymbol z))
=2^{-t}\prod_{j=1}^{t}\bigl(c_j-u(\boldsymbol z)\bigr).
\end{equation}
It then follows from \eqref{eq:pseudo-M} that
\begin{equation}\label{eq:h-Fourier}
h_g(\boldsymbol z)=f_t(\boldsymbol z)-(-2)^t\left(\prod_{j=1}^{t}\ell_j\right)\ind_{\boldsymbol z=\mathbf 1}.
\end{equation}
The Fourier coefficients of the polynomial kernel \(f_t\) will be thoroughly analyzed via a recurrence relation in Section~\ref{sec:fourier}, before incorporating the diagonal correction to invoke the inertia bound in Section~\ref{sec:inertia}.

\section{Boundary-layer spectral localization}\label{sec:fourier}

Throughout this section, we fix the dimension parameter \(m=n-1\) and maintain the shorthand notation 
\[
u=S+P, \quad c_j=n-4\ell_j
\]
introduced in Section~\ref{sec:setup}. 
For each $0\le r\le t$, put
\[
f_r(\boldsymbol z)=2^{-r}\prod_{j=1}^{r}(c_j-u(\boldsymbol z)), \quad \text{with } f_0=1.
\]
By convention, any polynomial of negative degree is understood to be identically zero.

\begin{lemma}[Fourier recurrence]\label{lem:recurrence}
For every \(0\le r\le t-1\) and any character index \(T\subseteq[m]\), the Fourier coefficients satisfy the recurrence relation
\begin{equation}\label{eq:recur}
\widehat f_{r+1}(T)=
\frac{c_{r+1}}{2}\widehat f_r(T)
-
\frac12\left(
\sum_{i=1}^{m}\widehat f_r(T\triangle\{i\})
+
\widehat f_r(T\triangle[m])
\right).
\end{equation}
\end{lemma}

\begin{proof}
By definition, the linear recurrence for the kernel reads \(f_{r+1} = \frac{1}{2}(c_{r+1} - u)f_r\). To evaluate its Fourier transform, we observe that the sum operator \(S\) and the product operator \(P\) act on the character basis via combinatorial propagation:
\[
S\chi_T=\sum_{i=1}^{m}\chi_{T\triangle\{i\}},
\qquad
P\chi_T=\chi_{T\triangle[m]}.
\]
Taking the normalized inner product of \(f_{r+1}\) with \(\chi_T\) and substituting these character actions immediately yields \eqref{eq:recur}.
\end{proof}

\begin{lemma}[Support localization and degree bounds]\label{lem:support-degree}
Let \(0\le r\le t\) and assume \(n>2r+1\). Then the Fourier support of \(f_r\) is strictly localized within the boundary layers
\[
\{T\subseteq [m]: |T|\le r\}\cup\{T\subseteq [m]: |T|\ge m-r+1\}.
\]
Furthermore, the coefficients exhibit the following polynomial dependencies on \(n\):
\begin{enumerate}[label=\textup{(\roman*)}]
    \item If \(|T|=s\le r\), then \(\widehat f_r(T)\) is a polynomial in \(n\) of degree at most \(r-s\).
    \item If \(|T|=m-s\) with \(0\le s\le r-1\), then \(\widehat f_r(T)\) is a polynomial in \(n\) of degree at most \(r-1-s\).
\end{enumerate}
In both cases, the coefficients of these polynomials are themselves polynomial functions of the distance parameters \(\ell_1,\dots,\ell_r\).
\end{lemma}

\begin{proof}
We proceed by induction on \(r\). For \(r=0\), the assertion holds trivially since \(\widehat f_0(\varnothing)=1\) and all other coefficients vanish. Suppose the hypothesis holds up to some \(r\), and assume \(n>2(r+1)+1\). We examine the structural propagation in \eqref{eq:recur} for \(f_{r+1}\).

First, consider a low-frequency layer where \(|T|=s\le r+1\). The unshifted term \(\frac{1}{2}c_{r+1}\widehat f_r(T)\), when nonzero, contributes a polynomial of degree at most \((r-s)+1 = r+1-s\). For the sum component, indices \(i\in T\) shift the frequency to layer \(s-1\), preserving the degree bound via \((r-(s-1)) = r+1-s\); indices \(i\notin T\) shift the frequency to layer \(s+1\), and summing over the \(m-s\) remaining coordinates yields a total degree contribution of at most \((r-(s+1)) + 1 = r-s\). Finally, the complementary product term \(\widehat f_r(T\triangle[m])\) addresses a high-layer character of size \(m-s\), which by the induction hypothesis is bounded by degree \(r-1-s\). Accumulating these terms confirms that the low-layer coefficient \(\widehat f_{r+1}(T)\) has degree at most \((r+1)-s\).

Second, consider a high-frequency layer where \(|T|=m-s\) for \(0\le s\le r\). The unshifted term \(\frac{1}{2}c_{r+1}\widehat f_r(T)\) yields a degree of at most \((r-1-s)+1 = r-s\). For the sum term, coordinates \(i\notin T\) map to size \(m-(s-1)\), and aggregation across these \(s\) positions maintains a degree bounded by \((r-1-(s-1)) = r-s\). Coordinates \(i\in T\) map to a higher-order layer and yield strictly lower degree terms. The complementary term \(T\triangle[m]\) lands on the low layer of size \(s\), contributing a polynomial of degree at most \(r-s\). This establishes the desired high-layer degree bound for \(f_{r+1}\).

Lastly, if the character index satisfies \(r+1<|T|<m-r\), then \(T\), its adjacent neighbors \(T\triangle\{i\}\), and its complement \(T\triangle[m]\) all strictly avoid the admissible support layers of \(f_r\). Thus, every term on the right-hand side of \eqref{eq:recur} vanishes identically, completing the inductive step.
\end{proof}

We are now positioned to establish the main boundary-layer localization theorem and explicitly compute the dominant coefficients governing the spectral boundary.

\begin{theorem}[Fourier boundary-layer localization]\label{thm:fourier}
Let \(t\ge1\), let \(\ell_1,\dots,\ell_t\) be distinct positive integers, and assume \(n>2t+1\). The Fourier coefficients of the distance kernel \(f_t\) defined in \eqref{eq:ft} satisfy the following exact algebraic profiles:
\begin{enumerate}[label=\textup{(\roman*)}]
\item For a low-layer character with \(|T|=s\) where \(0\le s\le t\):
\begin{equation}\label{eq:small}
\widehat f_t(T)=
\frac{(-1)^s}{2^t}\frac{t!}{(t-s)!}\,n^{t-s}
+R_{t,s}(n),
\end{equation}
where \(R_{t,s}(n)\) is a polynomial in \(n\) of degree at most \(t-s-1\). In the extremal case \(s=t\), the remainder satisfies \(R_{t,t}=0\).

\item For a high-layer character with \(|T|=m-s\) where \(0\le s\le t-1\):
\begin{equation}\label{eq:large}
\widehat f_t(T)=
\frac{(-1)^{s+1}}{2^t}\frac{t!}{(t-s-1)!}\,n^{t-1-s}
+R'_{t,s}(n),
\end{equation}
where \(R'_{t,s}(n)\) is a polynomial in \(n\) of degree at most \(t-s-2\). In the extremal case \(s=t-1\), the remainder satisfies \(R'_{t,t-1}=0\).

\item For all intermediate layers with \(t<|T|<m-t+1\), the coefficients vanish identically: \(\widehat f_t(T)=0\).
\end{enumerate}
The coefficients of the remainder polynomials \(R_{t,s}\) and \(R'_{t,s}\) depend polynomially on \(\ell_1,\dots,\ell_t\).
\end{theorem}

\begin{proof}
The vanishing property (iii) and the degree constraints on the remainders follow directly from Lemma~\ref{lem:support-degree}. It suffices to evaluate the leading coefficients by induction on the number of polynomial factors \(r\) for \(r=1,\dots,t\). For this purpose, we define the target leading coefficients on the low layers (\(|T|=s\)) and high layers (\(|T|=m-s\)) respectively by
\[
A_{r,s}=\frac{(-1)^s}{2^r}\frac{r!}{(r-s)!} \quad (0\le s\le r), \qquad B_{r,s}=\frac{(-1)^{s+1}}{2^r}\frac{r!}{(r-s-1)!} \quad (0\le s\le r-1),
\]
where coefficients outside these specified index ranges are treated as zero.

For the base case \(r=1\), setting \(\ell_1=\ell\) yields the kernel \(f_1(\boldsymbol z)=\frac{1}{2}(n-4\ell-u)\). Direct expansion yields the Fourier coefficients
\[
\widehat f_1(\varnothing)=\frac n2-2\ell,
\qquad
\widehat f_1(\{i\})=-\frac12 \quad (1\le i\le m),
\qquad
\widehat f_1([m])=-\frac12,
\]
with all other coefficients vanishing, matching the definitions of \(A_{1,s}\) and \(B_{1,s}\).

Now, assume the leading-term formulas hold for \(f_r\). Let \(|T|=s\le r+1\). According to the recurrence \eqref{eq:recur} and the degree bounds from Lemma~\ref{lem:support-degree}, the only terms capable of generating the maximal degree \(r+1-s\) are \(\frac{1}{2}c_{r+1}\widehat f_r(T)\) and the \(s\) neighboring terms arising from \(i\in T\). Consequently, the leading low-layer coefficient for \(f_{r+1}\) satisfies the relation
\[
\begin{cases}
\frac12A_{r,s}-\frac{s}{2}A_{r,s-1}, & 0\le s\le r,\\
-\frac{r+1}{2}A_{r,r}, & s=r+1.
\end{cases}
\]
An explicit algebraic verification confirms that
\[
\frac12A_{r,s}-\frac{s}{2}A_{r,s-1} = \frac{(-1)^s}{2^{r+1}}\frac{(r+1)!}{(r+1-s)!}=A_{r+1,s},
\]
and similarly \(-\frac{r+1}{2}A_{r,r} = \frac{(-1)^{r+1}}{2^{r+1}}(r+1)!=A_{r+1,r+1}\), validating the low-frequency profile for \(f_{r+1}\).

Next, let \(|T|=m-s\) for \(0\le s\le r\), where the maximal possible degree is \(r-s\). Invoking \eqref{eq:recur} and Lemma~\ref{lem:support-degree}, the dominant contributions are driven by \(\frac{1}{2}c_{r+1}\widehat f_r(T)\) (for \(s\le r-1\)), the \(s\) sum-induced terms where \(i\notin T\), and the complement block \(T\triangle[m]\) of size \(s\). The leading high-layer coefficient for \(f_{r+1}\) is thus given by
\[
\begin{cases}
\frac12B_{r,s}-\frac{s}{2}B_{r,s-1}-\frac12A_{r,s}, & 0\le s\le r-1,\\
-\frac r2B_{r,r-1}-\frac12A_{r,r}, & s=r.
\end{cases}
\]
Evaluating these expressions yields
\[
\frac12B_{r,s}-\frac{s}{2}B_{r,s-1}-\frac12A_{r,s} = \frac{(-1)^{s+1}}{2^{r+1}}\frac{(r+1)!}{(r-s)!}=B_{r+1,s},
\]
and \(-\frac r2B_{r,r-1}-\frac12A_{r,r} = \frac{(-1)^{r+1}}{2^{r+1}}(r+1)!=B_{r+1,r}\). This establishes the high-frequency formula for \(f_{r+1}\) and completes the induction.
\end{proof}

\begin{remark}\label{rem:mechanism}
The leading boundary-layer signs are universal. More precisely, although the kernel \(f_t\) is built from the prescribed distances \({2\ell_1,\dots,2\ell_t}\), these parameters enter the boundary-layer Fourier coefficients only through lower-degree remainder terms. The leading coefficients depend solely on \(t\) and the layer index \(T\), which is the mechanism behind the uniform spectral sign pattern for all fixed even distance sets of size \(t\).
\end{remark}

\section{The inertia argument and extremal consequences}\label{sec:inertia}

In this section, we deploy the algebraic machinery of spectral inertia to translate the boundary-layer Fourier localization established in Section~\ref{sec:fourier} into sharp extremal bounds. We begin by recalling Cvetkovi\'c's classical inertia bound~\cite{Cvetkovic}, a pivotal tool linking the independence number of a graph to the spectrum of its pseudo-adjacency matrices; see also~\cite{Brouwer-Haemers}.

\begin{proposition}[Cvetkovi\'c inertia bound]\label{prop:inertia}
Let \(\Gamma=(V,E)\) be a graph and let \(M\) be a real symmetric
pseudo-adjacency matrix of \(\Gamma\).  Let \(n_{\ge0}(M)\) and \(n_{\le0}(M)\)
denote the numbers of non-negative and non-positive eigenvalues of \(M\), counted
with multiplicity.  Then
\[
\alpha(\Gamma)\le \min\{n_{\ge0}(M),n_{\le0}(M)\}.
\]
\end{proposition}

\begin{proof}
Let \(I\subseteq V\) be an independent set of size \(k\).  By the definition of a
pseudo-adjacency matrix, the principal submatrix \(M[I,I]\) is the \(k\times k\)
zero matrix.  Cauchy's interlacing theorem implies that \(M\) has at least \(k\)
non-negative eigenvalues and at least \(k\) non-positive eigenvalues.  Thus
\[
k\le \min\{n_{\ge0}(M),n_{\le0}(M)\}.
\]
Taking the maximum over all independent sets proves the result.
\end{proof}

We are now ready to establish our main theorem.

\begin{proof}[Proof of \Cref{thm:main}]
A family \(\mathcal F\subseteq H\) with all pairwise distances restricted to \(\mathcal L\) is precisely an independent set in the distance graph \(\Gamma(n,\mathcal L)\). We apply the inertia bound (\Cref{prop:inertia}) to the pseudo-adjacency matrix \(M\) constructed in \eqref{eq:pseudo-M}.

By \eqref{eq:h-Fourier}, the Fourier coefficients of \(h_g\) evaluate to
\begin{equation}\label{eq:hhat}
\widehat h_g(T)=
\widehat f_t(T)-
\frac{(-2)^t\prod_{j=1}^{t}\ell_j}{2^{n-1}}.
\end{equation}
For all sufficiently large \(n\), \Cref{thm:fourier} combined with the exponentially small correction in \eqref{eq:hhat} yields the following sign pattern:
\[
\begin{array}{c|c}
\text{Frequency Layer} & \text{Sign of }\widehat h_g(T)\\
\midrule
|T|=s\le t & (-1)^s\\
|T|=m-s, \quad 0\le s\le t-1 & (-1)^{s+1}\\
t<|T|<m-t+1 & (-1)^{t+1}.
\end{array}
\]
Indeed, the signs in the first two rows are determined by the leading polynomial terms identified in \Cref{thm:fourier}. For the intermediate layers in the third row, \(\widehat f_t(T)=0\) identically, implying that the sign is solely driven by the trailing constant \(-(-2)^t\prod_j\ell_j\), which simplifies to \((-1)^{t+1}\).

If \(t\) is odd, the negative eigenspaces correspond to the low layers with odd \(s\) and the high layers with even \(s\); if \(t\) is even, the positive eigenspaces correspond to the low layers with even \(s\) and the high layers with odd \(s\). 
In both scenarios, the multiplicity of this specific sign class equals
\[
\sum_{\nu=0}^{t}\binom{m}{\nu}.
\]
Since \(m=n-1>2t\), we have \(\sum_{\nu=0}^{t}\binom{m}{\nu}\le 2^{m-1}\), ensuring that the complementary sign class has at least this multiplicity. Consequently, the minimum eigenvalue count satisfies
\[
\min\{n_{\ge0}(M),n_{\le0}(M)\}
=
\sum_{\nu=0}^{t}\binom{m}{\nu}
=
\sum_{\nu=0}^{t}\binom{n-1}{\nu}.
\]
Cvetkovi\'{c}'s inertia bound then yields the desired upper bound on \(|\mathcal F|\).

To demonstrate sharpness for the specific distance set \(\mathcal L_0=\{2,4,\dots,2t\}\), we construct an optimal family as follows. If \(t\) is even, we select the parity Hamming ball consisting of all even subsets of \([n]\) of size at most \(t\). If \(t\) is odd, we take the image under the parity-reversing involution \(A \mapsto A \triangle \{n\}\) of all odd subsets of \([n]\) of size at most \(t\). In both scenarios, the resulting family is strictly contained within \(H\), and the symmetric difference of any two distinct elements yields a positive even integer bounded above by \(2t\). By Pascal's identity, the cardinality of this family matches \(\sum_{i=0}^{t}\binom{n-1}{i}\) exactly, completing the proof.
\end{proof}

This fundamental upper bound on the even subcube immediately propagates to several adjacent settings. First, we record the dual result on the odd subcube.

\begin{corollary}[Odd-subcube analogue]\label{cor:odd}
Fix distinct positive integers \(\ell_1,\dots,\ell_t\), and let \(H'=\{A\subseteq[n]: |A|\text{ is odd}\}\) be the odd subcube. There exists a threshold \(n_0=n_0(\ell_1,\dots,\ell_t)\) such that, for all \(n\ge n_0\), any family \(\mathcal F\subseteq H'\) satisfying \(|A\triangle B|\in\{2\ell_1,\dots,2\ell_t\}\) for all distinct \(A,B\in\mathcal F\) obeys the bound
\[
|\mathcal F|\le \sum_{i=0}^{t}\binom{n-1}{i}.
\]
Moreover, this bound is sharp for the distance set \(\{2,4,\dots,2t\}\).
\end{corollary}

\begin{proof}
The parity-reversing involution \(\Phi(A) = A\triangle\{n\}\) defines a rigorous bijection from \(H'\) onto \(H\). Because \(\Phi\) preserves the underlying Hamming distances identically, the independent sets in the odd subcube map directly to independent sets of the same size in the even subcube. The result follows immediately from \Cref{thm:main}.
\end{proof}

Next, we handle arbitrary, mixed-parity distance constraints via parity decomposition and higher-dimensional liftings.

\begin{corollary}[Parity reduction for arbitrary distance sets]\label{cor:general}
Let \(\mathcal L\) be a finite set of positive integers, and write
\[
\mathcal L_{\mathrm{even}}=\mathcal L\cap2\NN,
\qquad
\mathcal L_{\mathrm{odd}}=\mathcal L\setminus\mathcal L_{\mathrm{even}},
\qquad
 e=|\mathcal L_{\mathrm{even}}|,
\qquad
\mathcal{L}_{\mathrm{odd}}+1=\{\ell+1:\ell\in \mathcal{L}_{\mathrm{odd}} \}.
\]
For all sufficiently large \(n\), any family \(\mathcal F\subseteq2^{[n]}\) with pairwise distances restricted to \(\mathcal L\) satisfies
\[
|\mathcal F|\le
\begin{cases}
\displaystyle \sum_{i=0}^{e}\binom{n-1}{i}, & \text{if } \mathcal L\subseteq2\NN,\\[12pt]
\displaystyle 2\sum_{i=0}^{e}\binom{n-1}{i}, & \text{if } \mathcal L\not\subseteq2\NN.
\end{cases}
\]
Moreover, if \(\mathcal L\not\subseteq2\NN \text{ and } \mathcal L_{\mathrm{odd}}+1\subseteq\mathcal L_{\mathrm{even}}\), then  the following tighter bound holds:
\[
|\mathcal F|\le\sum_{i=0}^{e}\binom{n}{i}.
\]
\end{corollary}

\begin{proof}
If \(\mathcal L\subseteq2\NN\), all admissible distances are strictly even, forcing all members of \(\mathcal F\) to share an identical parity. For \(e=0\), \(\mathcal F\) trivially contains at most one element; for \(e\ge1\), the first bound is a direct consequence of \Cref{thm:main} and \Cref{cor:odd}.

If \(\mathcal L\not\subseteq2\NN\), we decompose the family into its disjoint parity blocks: \(\mathcal F = \mathcal F_e \cup \mathcal F_o\), where \(\mathcal F_e = \mathcal F \cap H\) and \(\mathcal F_o = \mathcal F \cap H' \). Since pairwise distances within each isolated block must be even, they are strictly confined to \(\mathcal L_{\mathrm{even}}\). Applying the first case to \(\mathcal F_e\) and \(\mathcal F_o\) independently and summing the bounds yields the second estimate.

Finally, under the nested parity condition \(\mathcal L_{\mathrm{odd}}+1\subseteq\mathcal L_{\mathrm{even}}\), we map \(\mathcal F\) into the even subcube of a higher dimension \([n+1]\) via the canonical embedding:
\[
A^*=\begin{cases}
A, & \text{if } |A|\text{ is even},\\
A\cup\{n+1\}, & \text{if } |A|\text{ is odd}.
\end{cases}
\]
For any pair \(A,B\in\mathcal F\) of matching parity, \(|A^*\triangle B^*|=|A\triangle B|\in \mathcal L_{\mathrm{even}}\). For a pair of opposite parity, the embedding introduces an additional offset, yielding \(|A^*\triangle B^*|=|A\triangle B|+1 \in \mathcal L_{\mathrm{odd}}+1 \subseteq \mathcal L_{\mathrm{even}}\). Thus, the lifted family \(\mathcal F^*\) resides entirely in the even subcube of \([n+1]\) with distances restricted to \(\mathcal L_{\mathrm{even}}\). Applying \Cref{thm:main} in dimension \(n+1\) yields the third bound.
\end{proof}

By substituting \(\mathcal L=\{1,2,\dots,d\}\) into \Cref{cor:general},
we recover the classical Kleitman bounds, with the sharp threshold discussed in Section~\ref{sec:thresholds}. \Cref{cor:general}, applied to
\(\mathcal L=\{2s+1,\dots,2t\}\) and \(\mathcal L=\{2s+1,\dots,2t+1\}\), recovers the consecutive-distance theorem of Huang--Klurman--Pohoata
\cite[Theorem~1.2]{HKP}. We conclude by showing that our framework provides an efficient spectral proof of the maximality of the distance set \(\{2,4,\dots,2t\}\).

\begin{corollary}[Maximality of \(\{2,4,\dots,2t\}\)]\label{cor:D0-max}
Let \(D\) be an arbitrary fixed set of \(t\) positive integers, and let \(A(n,D)\) denote the maximum size of a family \(\mathcal F\subseteq2^{[n]}\) whose pairwise distances lie in \(D\). Setting \(D_0=\{2,4,\dots,2t\}\), then for all sufficiently large \(n\), we have
\[
A(n,D)\le A(n,D_0)=\sum_{i=0}^{t}\binom{n-1}{i}.
\]
Furthermore, if \(D\) contains at least one odd integer, the bound tightens asymptotically to
\[
A(n,D) \le 2\sum_{i=0}^{t-1}\binom{n-1}{i} = o(n^t),
\]
whereas the extremal configuration satisfies \(A(n,D_0)\sim \frac{1}{t!}n^t\).
\end{corollary}

\begin{proof}
For \(D_0=\{2,4,\dots,2t\}\), parity splitting reduces the problem to a
single parity class. Thus \Cref{thm:main} and \Cref{cor:odd} give the upper
bound, while the parity-ball construction in the proof of \Cref{thm:main}
attains it. Hence
\(
A(n,D_0)=\sum_{i=0}^{t}\binom{n-1}{i}.
\)
If \(D\subseteq2\NN\), then \(|D_{\mathrm{even}}| = t\), and \Cref{cor:general} yields \(A(n,D)\le \sum_{i=0}^{t}\binom{n-1}{i} = A(n,D_0)\). If \(D\) contains an odd integer, the size of its even restriction is strictly bounded by \(|D\cap2\NN|\le t-1\). Invoking the second case of \Cref{cor:general} gives \(A(n,D)\le2\sum_{i=0}^{t-1}\binom{n-1}{i}\). Since this is a polynomial in \(n\) of degree at most \(t-1\), it is asymptotically dominated by the \(n^t\) growth of \(A(n,D_0)\), establishing the proof.
\end{proof}

\begin{remark}\label{rem:DGLOZ-relation}
While \Cref{cor:D0-max} efficiently recovers the maximality statement of the Dong--Gao--Liu--Ouyang--Zhou theorem~\cite{DGLOZ} for fixed \(t\)-distance sets, our proof does not capture their refined uniqueness or gap phenomena. Instead, it provides a highly decoupled, spectral alternative that reaches the global maximum via a distinct, analytic route.
\end{remark}

\section{Thresholds and final remarks}\label{sec:thresholds}

The validity of our spectral framework hinges on two dimension requirements. 
First, the low and high boundary layers must be separated by a nonempty middle range \[ t<|T|<m-t+1, \] which is equivalent to \(m>2t\), or \(n>2t+1\).
Second, on the boundary layers, the leading terms in \Cref{thm:fourier} must dominate the lower-degree remainders and the exponentially small diagonal correction in \eqref{eq:hhat}. Since all boundary coefficients are governed by the explicit recurrence \eqref{eq:recur}, the requisite lower bound on \(n\) is computationally effective for each fixed distance set.

For the distance set
\[
\mathcal L_0=\{2,4,\dots,2t\},
\]
the threshold \(n>2t+1\) is sharp. Indeed, for \(x=2k\), the distance
kernel simplifies to
\[
f(x)=\prod_{j=1}^{t}(x-2j)=2^t t!\binom{k-1}{t}.
\]
Let \(M_k\) denote the distance-\(2k\) adjacency matrix on the even subcube.
For this distance set \(\mathcal L_0\), the pseudo-adjacency matrix \(M\)
defined in \eqref{eq:pseudo-M} expands as
\[
M=
2^t t!\sum_{k=t+1}^{\lfloor n/2\rfloor}
\binom{k-1}{t}M_k .
\]
Consequently, after removing the positive scalar \(2^t t!\), this is precisely
the spectral matrix used for the parity-ball bound in~\cite[Section~5.2]{DGLOZ}.
The explicit eigenvalue computation in that section shows that the sign
pattern required for the inertia argument holds precisely 
throughout the
range
\[
m=n-1>2t.
\]
This threshold cannot be lowered. At the endpoint \(m=2t\), every nonzero even distance in the even subcube belongs to
\(\{2,4,\dots,2t\}\). Hence the whole even subcube is admissible for
\(\mathcal L_0\), and has size \(2^{2t}\), which is larger than
\[
\sum_{i=0}^{t}\binom{2t}{i}.
\]
Thus the bound in \Cref{thm:main} is false at \(n=2t+1\), and the threshold
\(n>2t+1\) is best possible for this distance set.

When translated back to the classical Kleitman theorem through
\Cref{cor:general}, this condition yields the classical hypothesis \(n>d\).
Indeed, if the diameter is odd, \(d=2k+1\), the even-subcube threshold
\(n> 2k+1\) is exactly \(n>d\); if the diameter is even, \(d=2k\), the
lifted dimension condition \(n+1> 2k+1\) likewise reduces to \(n>d\).

\section*{Acknowledgments}
Xiaochen Zhao thanks Professor Zixiang Xu and Professor Hao Huang for helpful early discussions.

\bibliographystyle{plain}
\bibliography{references}
\end{document}